\documentclass{amsart}
\usepackage{graphicx}
\pdfoutput=1 

\newtheorem{theorem}{Theorem}[section]
\newtheorem{lemma}[theorem]{Lemma}

\theoremstyle{definition}
\newtheorem{definition}[theorem]{Definition}

\theoremstyle{remark}

\numberwithin{equation}{section}

\begin{document}

\title{Generalized Small Cancellation Presentations for Automatic Groups}

\author{Robert H. Gilman}
\address{Department of Mathematical Sciences\\Stevens Institute of Technology}
\email{rgilman@stevens.edu}
\thanks{Partially supported by NSF grant 1318716}



\date{June 28, 2014}

\begin{abstract}
By a result of Gersten and Short finite presentations satisfying the
usual non-metric small cancellation conditions present biautomatic groups.
We show that in the case in which all pieces have length one, a generalization
of the C(3)-T(6) condition yields a larger collection of biautomatic groups.
\end{abstract}

\maketitle

\newcommand{\inv}{^{-1}}
\newcommand{\ra}{\rightarrow} \newcommand{\tto}{\buildrel * \over 
\rightarrow }
\newcommand{\edge}[1]{\buildrel #1 \over \rightarrow } 
\newcommand{\abs}[1]{\vert #1\vert} \newcommand{\Abs}[1]{\Vert #1 \Vert}
\newcommand{\set}[1]{\{ #1 \}}
\newcommand{\subgroup}[1]{\langle #1 \rangle}
\newcommand{\ovr}[1]{\overline #1}
\newcommand{\G}{\Gamma}
\newcommand{\K}{\mathcal K}
\newcommand{\PP}{P^{(2)}}

\section{Introduction}
Gersten and Short show in~\cite{GS1} that groups satisfying any of the usual non-metric small cancellation conditions C(6), or C(4)-T(4), or C(3)-T(6) are biautomatic. Most of their argument with respect to C(3)-T(6) applies verbatim to a more general situation and yields small cancellation type presentations for a larger class of automatic groups. These presentations are essentially finite prees (in the sense of Frank Rimlinger~\cite{Ri}) which satisfy two axioms introduced by Reinhold Baer~\cite{Ba}.

A pree, defined below, is a set with a partial multiplication satisfying certain axioms. An example is an amalgam of two groups; the multiplication is defined for pairs of elements within each group but is not defined for pairs which are not within one group or the other. Amalgams obey an additional axiom which make them pregroups as defined by John Stallings~\cite{St1}. See the recent survey~\cite{GL2012} for an introduction to prees and their axioms.

The axioms we employ, A(4) and A(5), are given in Section~\ref{prees}. Their relation to the usual small cancellation conditions is discussed at the beginning of Section~\ref{cancellation}, where we develop a variation of small cancellation theory. There are many such variations, e.g.~\cite{J1, J2, J3, J4, MW1, MW2, O}. The one which which seem closest to ours is by Uri Weiss~\cite{W}, who shows that V(6) groups with all pieces of length 1 are biautomatic. V(6) denotes a generalization of the C(6) and C(4)-T(4) conditions. 

Section 4 contains the proof of our main theorem, Theorem~\ref{A45}, which says that the universal groups of finite prees satisfying Axioms A(4) and A(5) are biautomatic. As every finite group satisfies these two axioms and is its own universal group, the class of groups covered by Theorem~\ref{A45} includes all finite groups. Finite C(3)-T(6) groups, on the other hand, are all cyclic~\cite{EH}. It would be interesting to have other examples of groups from Theorem~\ref{A45} which are not C(3)-T(6).

\section{Prees}\label{prees}

\begin{definition}\label{pree}
A pree is a set equipped with a partial multiplication which affords an identity, denoted $1$, two-sided inverses, and the following form of associative law.
\end{definition}

It is not hard to show that inverses are unique in a pree. Further if $ab=c$, then $b = a\inv c$ etc. It follows that products $ab=c$ may be thought of as triangles. More precisely the six products obtained by reading the edge labels around the boundary of the triangle on the left-hand side of Figure~\ref{triangles} in either direction and starting at any vertex are defined if any one of them is. As usual an edge read against its orientation contributes the inverse of its label. 
\begin{figure}
\includegraphics[width=.5\textwidth]{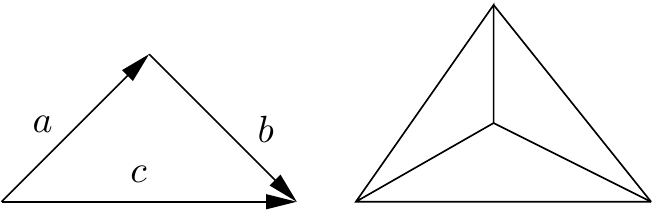}
\caption{The product $ab=c$ (left) and the associative law (right).\label{triangles}}
\end{figure}
\begin{definition}[The associative law]\label{associative}
If $ab$ and $bc$ are defined, then $(ab)c$ is defined if and only if $a(bc)$ is; and when they are defined, $(ab)c=a(bc)$.  
\end{definition}

The associative law is equivalent to a geometric condition, namely that if three triangles fit around a common vertex as in the right-hand side of Figure~\ref{triangles}, then the perimeter is also a valid triangle in 
$P$.

\begin{definition} The universal group of a pree $P$ is the group $U(P)$ 
with generators $P$ and relators $abc\inv$ for every product $ab=c$ 
defined in $P$. We write $\overline a$ for the image of $a$ in $U(P)$.
Elements of $U(P)$ are represented by words over the alphabet $P$. Note that $1$ is a letter in this alphabet.
\end{definition}
\begin{definition}\label{reducible}
A word over $P$ is reducible if the product of two successive letters is defined in $P$. Otherwise the word is called irreducible or reduced.
\end{definition}

It is straightforward to use Tietze transformations to change any finite
presentation into one given by a finite partial multiplication table, and with more
Tietze transformations one can insure that this multiplication table
satisfies the conditions of a pree. Thus every finitely presented group is the universal group of a finite pree, and the pree affords a finite presentation for the group. 

\begin{theorem}\label{embedding} It is undecidable whether or not the canonical morphism $\pi: P \to U(P)$ from a pree $P$ to its universal group is an embedding.
\end{theorem}

This theorem is a special case of a result of Trevor Evans~\cite{Ev} 
which says that if the embedding problem is solvable for a class of 
finite partial algebras, then the word problem is solvable for the 
corresponding class of algebras.

Now we present the additional axioms we need. There are two of them, and they have the geometric incarnations given in Figure~\ref{spree}

\begin{figure}
\includegraphics[width=.9\textwidth]{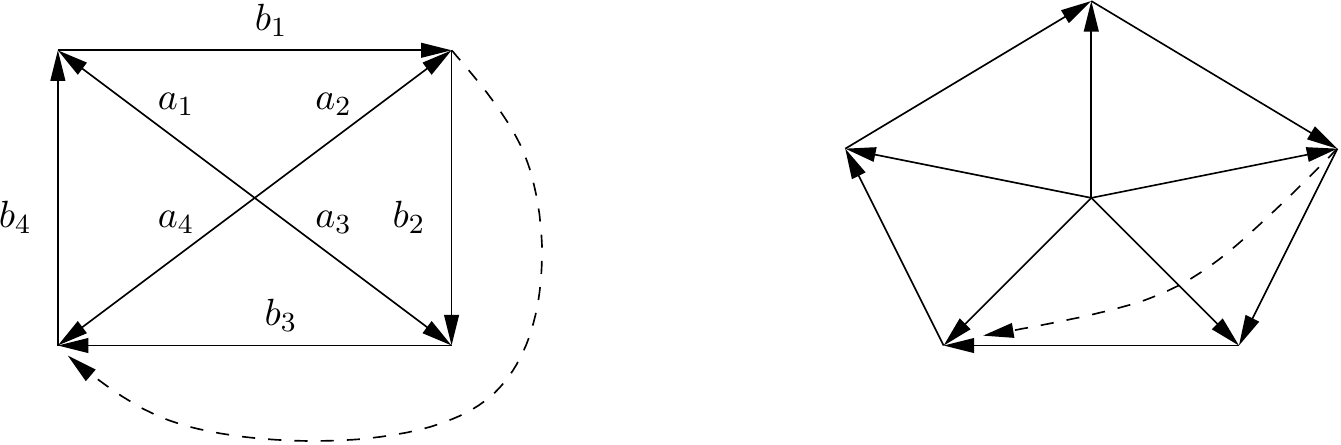}
\caption{Axioms A(4) and A(5) with $b_2b_3$ defined. \label{spree}}
\end{figure}

\begin{description}
\item [Axiom A(4)] If the products $a_1\inv a_2=b_1$, $a_2\inv a_3=b_2$, $a_3\inv a_4=b_3$, and $a_4\inv a_1=b_4$ are defined, then at least one of the products $b_ib_{i+1}$ (indices read modulo 4) is defined. 
\item [Axiom A(5)] If the products $a_1\inv a_2=b_1$, $a_2\inv a_3=b_2$, $a_3\inv a_4=b_3$, $a_4\inv a_5=b_4$ and $a_5\inv a_1=b_5$ are defined, then at least one of the products $b_ib_{i+1}$ (indices read modulo 5) is defined.
\end{description}

\section{Small Cancellation Theory}\label{cancellation}

We require a simple extension of small cancellation theory for diagrams constructed from triangles

 All relations in the presentation $P$ for $U(P)$ have length 3, and there are no pieces of length 2. Indeed $ab$ can occur in at most one relator, namely the relator $abc\inv$ corresponding to a product $ab=c$ defined in $P$. Thus the small cancellation condition C(3) holds.

Similarly the condition T(6) says that the perimeter of any diagram formed by fitting 2, 3, 4 or 5 triangles around a common vertex (as in Figure~\ref{spree} and the right-hand side of Figure~\ref{triangles}) contains a subword $aa\inv$ for some $a\in P$. Axioms A(4) and A(5) on the other hand together with the Associative law enforce the weaker condition that the perimeter is reducible.
 
\begin{definition}
A diagram is a planar directed labeled graph whose boundary is a simple 
closed curve and whose faces are triangles each with 3 distinct 
vertices. The label of the boundary of each face is a relator in $P$. 
The number of triangular faces of a diagram is its area. The boundary word, determined up to cyclic permutation and inverse, is the label of the boundary path. Boundaries of triangles are sometimes called perimeters. For brevity boundary words are sometimes themselves called boundaries.
\end{definition}

The label of a path in a diagram is the product of the labels of edges 
except that an edge traversed against its orientation contributes the 
inverse of its label. We may change the orientation of an edge as long 
as we invert its label. To remove a triangle with one or more edges on 
the boundary we remove the edges on the boundary except we keep any 
vertices which began with degree more than 2.

\begin{lemma} If $D$ is a diagram with area greater than one, there is a 
triangle in $D$ with one or more edges on the boundary such that 
removing that triangle results in a diagram.
\end{lemma}

\begin{proof}
If there is a triangle with two edges on the boundary, then because the boundary is a simple closed curve, the common vertex has degree 2, and the triangle may be removed. Likewise a triangle with only one edge on the 
boundary and the opposite vertex in the interior of the diagram may be removed. 

Suppose 
then that all triangles meeting the boundary have exactly one edge and 
its opposing vertex on the boundary. For each such triangle the 
complement of the diagram is divided into two parts. Since the triangle 
has two edges not on the boundary, both parts have positive area. Pick a 
triangle such that one of the complementary parts has minimum possible 
area. Since that area is positive, it must contain another triangle 
meeting the boundary. But then that triangle would give a complementary 
part of smaller area. Thus this last case cannot arise.
\end{proof}

The preceding lemma has the following consequence.

\begin{lemma}\label{recursive}
Every diagram is generated by starting with a single triangle and attaching subsequent triangles in the following way. A triangle is attached to a diagram by identifying one or two edges of the triangle with an edge or two successive edges respectively of the boundary of the diagram so that identified edges have the 
same label and are oriented in the same direction or have inverse labels 
and are oriented in opposite directions. In addition two vertices of a
triangle may not be identified when it is attached.
\end{lemma}

\begin{lemma}\label{boundary} If a diagram has two triangles with the same vertices, then it has an internal vertex of degree two.
\end{lemma}

\begin{proof}
Let $D$ be a diagram. If $D$ has area 1, there is nothing to prove. Otherwise $D$ is formed by attaching a triangle $T$ to a diagram $D'$ of smaller area. By induction we may assume $D'$ does not have two triangles with the same vertices. Consequently the vertices of $T$ are also the vertices of some triangle $T'$ of $D'$. But then these vertices must be the vertices of a path of length 2 on the boundary of $D'$. It follows that the middle vertex is an internal vertex of degree 2 in $D$.
\end{proof}

\begin{definition}
 A minimal diagram is one which has minimum area among all diagrams with the same boundary word (up to cyclic permutation and inverse.).
\end{definition}

\begin{lemma}\label{minimal}
Let $D$ be a minimal diagram with boundary $w$. If $D$ has an internal vertex of degree 2, then $D$
consists of two triangles joined along two edges and $w=aa\inv$.
\end{lemma}

\begin{proof}
It is easy to see that no vertex of any diagram has degree 1. If $p$ is an internal vertex of degree 2, then $p$ is a vertex of two triangles both of which contain the two edges incident to $p$. Hence both edges opposite $p$ have label $c=a\inv b$. Unless both of these edges are on the boundary of $D$, they can be combined to make a diagram of smaller area. If they are on the boundary, they must be the whole boundary because the boundary of $D$ is a simple closed curve. Thus $D$ consists of two triangles joined along two edges, and $w=cc\inv$.
\end{proof}

\begin{lemma}\label{le:curvature} Consider any diagram, and let $d(p)$ be the degree of each 
vertex $p$, then
\begin{equation} \label{eq:curvature}
\sum^{\circ} 4-d(p) =6 + \sum^{\bullet} d(p)-6
\end{equation}
where the first sum is over all boundary vertices and the second is over 
all internal ones.
\end{lemma}
\begin{proof}
Argue by induction on area.
\end{proof}

Let $L(P)$ be the language of all labels of boundaries of diagrams. Each 
diagram determines multiple labels as we may start at any vertex 
on the perimeter and proceed clockwise or counterclockwise. Notice that 
because we do not identify vertices when building a diagram by attaching 
triangles, the boundary of every diagram has length at least two.

\begin{lemma} $L(P)$ is the set of all words $w$ over the alphabet $P$ of length at least 2 such that $\overline w$ is the identity in $U(P)$.
\end{lemma}

\begin{proof}
$U(P)$ is the quotient of the free semigroup $S$ over $P$ by the congruence $\sim$ with generators $ab\sim c$ for all products $ab=c$ defined in $P$. Say that $v, w\in S$ are {\em neighbors} if one is obtained from the other by replacing a subword $ab$ with $c$. Observe that if  $v$ and $w$ are neighbors of length at least 2, and $D$ is a diagram for $w$, then a diagram for $v$ can be obtained by attaching a triangle to $D$. 

Suppose $|w|\ge 2$ and $w$ defines the identity in $U(P)$. Then there is a sequence of neighbors 
\[ w=u_1, u_2, \ldots, abc\inv \]
where $ab=c$ is any product defined in $P$. As all elements of the sequence 
\[ w, 1w, 1u_1, 1u_2, \ldots, 1abc\inv, abc\inv\]
have size at least 2, it follows that there is a diagram for $w$.

Conversely if $w$ is a perimeter label for a diagram $D$, then by removing triangles we can reduce $D$ to a single triangle $T$. The corresponding sequence of neighbors begins with $w$ and ends with a label of the perimeter of $T$. It follows that $\overline w = 1$.
\end{proof}

\section{Biautomaticity}\label{biautomaticity}

We consider a fixed pree $P$ satisfying Axioms A(4) and A(5) and show that it is biautomatic. Much of our argument follows that in~\cite{GS1}. 

\begin{lemma}\label{d6} If $D$ is a minimal diagram without internal vertices of degree 2, then all internal vertices have degree at least 6.
\end{lemma}

\begin{proof}
If an internal vertex $p$ has degree 3, then by the associative law from 
Definition~\ref{associative} we can remove $p$ and all incident edges and still have a diagram over $P$. Likewise if $p$ has degree 4 or 5, then Axioms A(4) and A(5) guarantee that one can discard $p$ and construct a diagram of smaller area by adding one or two internal edges.
We must check that each new triangle created by the addition of edges has all vertices distinct. But if it did not, $D$ would have two triangles sharing the same three vertices, contrary to Lemma~\ref{boundary}.
See Figure~\ref{spree}.
\end{proof}

\begin{lemma} \label{galleries} Let $w$ be a word of length at least 3 
defining the identity in $U(P)$, and let $D$ be a diagram of minimal area 
with boundary $w$. Let $\delta_2$ and $\delta_3$ be the number of 
vertices of degree 2 and 3 respectively on the boundary of $D$, and let 
$\delta_5$ be the number of degree greater than 4, then $2\delta_2 + 
\delta_3 \ge 6 + \delta_5$. Further, if equality holds, then all 
internal vertices have degree 6.
\end{lemma}

\begin{proof}
Each vertex of degree 2 contributes 2 to the left-hand side of the 
Equation~\ref{eq:curvature} in Lemma~\ref{le:curvature}. Likewise each vertex of degree 3 contributes 
1, and vertices of degree greater than 4 contribute -1 or less. As all 
internal vertices have degree at least 6 by Lemmas~\ref{minimal} and~\ref{d6}, 
Equation~\ref{eq:curvature} yields the desired result.
\end{proof}

\begin{lemma} \label{45} Words of length 4 or 5 which define the 
identity in $U(P)$ are reducible.
\end{lemma}

\begin{proof}
Let $w$ be such a word and apply Lemma~\ref{galleries} to a minimal 
diagram for $w$. If $\delta_2\ge 2$, then some boundary vertex in the 
interior of $w$ has degree 2, and we are done. The only other possibility 
is $\abs w = 5$, $\delta_2=1$, and $\delta_3=4$. But it is easy to see 
that there is no diagram with these parameters.
\end{proof}

\begin{theorem}\label{A45} If $P$ is a pree satisfying Axioms A(4) and 
A(5), then $P$ embeds in $U(P)$, and the multiplication in $P$ is 
induced by the multiplication in $U(P)$.
\end{theorem}
\begin{proof}
The theorem is proved by the same small cancellation argument used 
in~\cite{GS1}. Let $w$ be a word of length at least two over $P$ which 
defines the identity in $U(P)$ and consider a diagram $D$ of minimum 
area for $w$. We must show that if $w=ab$, then $b=a\inv$; and if 
$w=abc$, then $ab=c\inv$ in $P$. By Lemmas~\ref{minimal} and~\ref{d6} we may assume 
that $w=abc$ and all internal vertices of $D$ have degree at least 6. By 
Lemma~\ref{galleries} there are either three boundary vertices of degree 
2 or at least 4 boundary vertices. It follows that $D$ is a triangle. 
\end{proof}

The proof of the next theorem is given in a sequence of lemmas. It is 
useful to keep in mind the following example
$$
P=\set{(0,0),(0,1), (0,-1), (1, 0), (-1,0), (1,1), (-1,-1) }
$$
with $U(P)=Z\times Z$. The partial multiplication in $P$ is inherited from the usual addition in $Z\times Z$.

\begin{theorem}If $P$ is a finite pree satisfying Axioms A(4) and A(5), 
then $U(P)$ is biautomatic.
\end{theorem}

First we observe that Lemma~\ref{galleries} implies that in any diagram 
$D$ of minimum area there must be a number of intervals along the 
boundary consisting of two vertices of degree 2 or 3 either adjacent to 
each other or separated by vertices of degree 4. As in~\cite{GS1} we 
refer to these intervals as galleries. See Figure~\ref{gallery}.
\begin{figure}
\includegraphics[width=.7\textwidth]{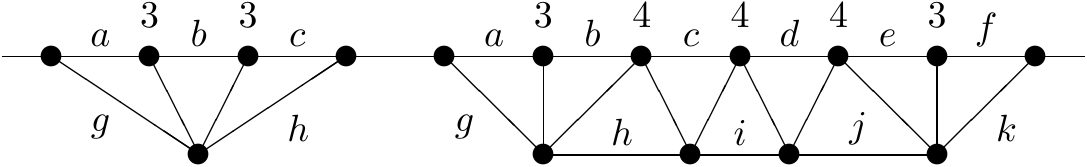}
\caption{Some galleries. \label{gallery}}
\end{figure}
Note that some of the vertices in this figure might be identified in 
$D$. In other words there is a morphism of diagrams from the gallery as 
illustrated to $D$. Under this morphism the labeled vertices have 
distinct images and their degrees are preserved.

The minimum number of galleries varies as in Table~\ref{g} depending on 
the value of $\delta_2$.

\begin{table}
\begin{tabular}{ccc}
$\delta_2$ & $\delta_3$ & Galleries\\ \hline
 $\ge 3$ &$\ge 0$& 3 \\
2 & $\ge 2$ & 4 \\
1 & $\ge 4$ & 5 \\
0 & $\ge 6$ & 6 \\
\end{tabular}
\caption{Effect of $\delta_2$ on $\delta_3$ and the minimum number of galleries. \label{g}}
\end{table}

Now we augment Definition~\ref{reducible}.

\begin{definition} A word $w=a_1\cdots a_n$ over $P$ is irreducible if  
for all $i$ the product $a_ia_{i+1}$ is not defined in $P$, and it is 
strongly irreducible if in addition it cannot be shortened by attaching 
to it a gallery of type $34^k3$ as in Figure~\ref{gallery} and replacing
$abc$ by $gh$ or $abcdef$ by $ghijk$ etc. Reductions of this latter type are called strong reductions.
\end{definition}

Both reductions and strong reductions shorten a word's length by 1. 
There are also reductions corresponding to other types of galleries. 
Reductions corresponding to galleries of type $24^k3$ shorten the length 
of a word by 2.

Since galleries are diagrams over $P$, it is clear that the above 
reductions do not change the element of $U(P)$ represented by a word, 
over $P$. Thus every word can be reduced to a strongly irreducible word 
representing the same group element.

\begin{lemma}\label{geodesic} A word $w$ over $P$ is the label of a 
geodesic in $U(P)$ if and only if it is strongly reduced and not the 
word $1$. Further the set of strongly irreducible words is a regular 
language.
\end{lemma}
\begin{proof}
Clearly the set of irreducible words not equal to $1$ is regular. 
Further it is straightforward to construct a finite automaton over the 
alphabet $P\times P \cup P\times \set{\$}$ which accepts a pair of words 
$(w,v\$)$ if and only if $w$ and $v$ are words over $P$ and $v$ is 
obtained from $w$ from a single reduction corresponding to attaching a 
gallery to $w$. Here $\$$ is a padding symbol not in $P$. The triangles 
of $P$ can serve as the vertices of a suitable automaton. It follows by 
standard techniques that the set of all words over $P$ which are strongly 
reducible forms a regular language. Hence its complement is regular and 
the intersection of the complement with the irreducible words not equal 
to $1$ is regular too. Thus the second assertion of the lemma holds.

Clearly geodesics are strongly irreducible lest there be a shorter word 
denoting the same group element. Suppose that $w$ is strongly 
irreducible but not geodesic. In particular $w \ne aa\inv$. Also $\abs 
w\ne 1$ by Theorem~\ref{A45}, so $\abs w \ge 3$. By our hypothesis there 
must be a shorter word $v$ representing the same group element as $w$. 
Thus there will be a diagram $D$ of minimum area with boundary label 
$wv\inv$. Pick $v$ so that $D$ has smallest possible area.

Consider the possibilities in Table~\ref{g}. There cannot be 3 vertices 
of degree 2, because then at least one of them would be in the interior 
of $w$ or $v$ contradicting the fact that both are strongly irreducible. 
Likewise there can be no gallery all of whose vertices of degree 3 or 4 
lie in the interior of $w$ or in the interior of $v$. Thus there are 4 
galleries and 2 vertices of degree 2. The only way they all fit into $D$ 
is if $\abs v \ge 1$, and the two boundary vertices separating the 
boundary segments with labels $w$ and $v$ are of degree 2 and are each 
part of two galleries. Each gallery has a vertex of degree 2 or 3 at its 
other end. Since these vertices are in the interior of $w$ or $v$, they 
must be of degree 3. In particular $\abs v \ge 2$.

Let $p$ and $q$ be the vertices separating $w$ and $v$ so that the 
boundary of $D$ consists of a path with label $w$ from $p$ to $q$ and a 
path with label $v$ also from $p$ to $q$. Let $w=a_1a_2\cdots w_m$, and 
$v=b_1b_2\cdots b_n $, and denote by $r$ the boundary vertex between 
$a_1$ and $a_2$.  By the previous paragraph there is a gallery of type 
$24^k3$ attached to the boundary segment with label $a_1\inv b_1\cdots 
b_{k+2}$. This gallery affords a reduction of $a_1\inv b_1 \cdots 
b_{k+2}$ to $u=c_1\cdots c_{k+1}$. It follows that the path from $p$ to 
$q$ with label $a_1c_1\cdots c_{k+1}b_{k+3}\cdots b_m$ is a geodesic. If 
this path is $w$, then we are done. Otherwise there are paths from $r$ 
to $q$ with labels $a_2\cdots a_m$ and $c_1\cdots c_{k+1}b_{k+3}\cdots 
b_n$. The first label is clearly strongly irreducible, and the second is 
a geodesic. By induction on length, $m-1=n-1$.
\end{proof}

In order to show that $U(P)$ is biautomatic, we must find a set $L$ of 
words over $P$ which maps onto $U(P)$ and such that for some constant 
$K$ two paths which begin and end a distance at most one apart and which 
have labels in $L$ $K$-synchronously fellow travel. $L$ will consist of some
geodesics, but not all of them. At this point in~\cite{GS1} use is made 
of the C3-T6 small cancellation hypothesis, which need not hold in our situation. We must proceed differently.

\begin{figure}
\includegraphics[width=.7\textwidth]{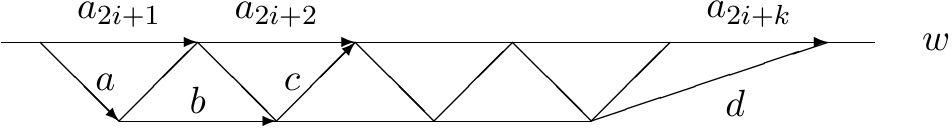}
\caption{ Defining $L$ \label{maximal}}
\end{figure}

\begin{definition} \label{combing} Let $L$ be the set of strongly 
reduced words $w=a_1\cdots a_n$ over $P$ with the property that whenever 
there is a valid diagram of the type illustrated in Figure~\ref{maximal} 
with $k\ge 3$, then the product $bc$ is defined.
\end{definition}

\begin{lemma}
$L$ is a regular set which maps onto $U(P)$.
\end{lemma}
\begin{proof}
The set of words which admit a valid diagram as in Figure~\ref{maximal} 
such that $bc$ is not defined is clearly regular. It follows from this 
observation and from Lemma~\ref{geodesic} that $L$ is regular. To show 
that $L$ maps onto $U(P)$ let $g$ be any element of $U(P)$ and pick a 
geodesic $w=a_1\cdots a_n$ from $1$ to $g$ with the property that the 
number of words $ef$ representing the same group element as $a_1a_2$ is 
maximal, and subject to that the number representing the same group 
element as $a_3a_4$ is maximal, and so forth.

Suppose $w$ admits a diagram as in Figure~\ref{maximal} for which $bc$ 
is not defined. Let $ef$, representing the same group element as 
$a_{2i+1}a_{2i+2}$. As $efc\inv b\inv a\inv$ represents the identity in 
$U(P)$, it is reducible by Lemma~\ref{45}. On the other hand $b\inv 
a\inv$ and $ef$ are subwords of geodesics and so are irreducible. 
Further $bc$ is irreducible by assumption. Thus $fc\inv=f'$ for some 
$f'\in P$. We conclude that $ef'$ represents the same group element as 
$ab$.

But no $ef'$ obtained as above can have $e=a$ and $f'=b$. For if so, 
$1=efc\inv b\inv a\inv = afc\inv b\inv a\inv$ would imply that $bc=f$ in 
$P$. Thus the geodesic obtained by substituting $ab\cdots d$ for 
$w_{2i+1}\cdots w_{2i+k}$ in $w$ contradicts the choice of $w$.
\end{proof}

It remains to show that for some constant $K$ two paths which begin and 
end a distance at most one apart and which have labels in $L$, 
$K$-synchronously fellow travel. Suppose $w = a_1\cdots a_n$ and $v = 
b_1\cdots b_m$ are labels of two such paths. Let $g_i$ be the group 
element reached by $a_1\cdots a_i$, and define $h_i$ likewise with 
respect to $v$. We claim that for each $i$ there exists a $j$ such that 
$\rho(g_i,h_j)$, the distance in $U(P)$ between $g_i$ and $h_j$, is at 
most 2; and likewise with $i$ and $j$ reversed. Since $w$ and $v$ are 
both geodesics beginning a distance at most one apart, it follows by a 
straightforward argument that $\abs {i-j}\le 3$; and hence that $K=5$ 
suffices.

We use induction on $m+n$. It $m$ and $n$ are both at most 3, then the 
desired conclusion is immediate. Thus we assume $m\ge 4$ or $n\ge 4$. 
Since $w$ and $v$ are geodesics, it follows that $m\ge 2$ an $n\ge 2$.
Because of the way $L$ is defined, it suffices to show that 
$\rho(g_2,h_2)\le 1$.

By hypothesis there are letters $a,b\in P$ such that $awb\inv v\inv$ 
represents the identity in $U(P)$. Note that $a=1$ and $b=1$ are 
possible. Consider a corresponding diagram $D$ of minimum area with 
vertices labeled as in Figure~\ref{rectangle}.
\begin{figure}
\includegraphics[width=.5\textwidth]{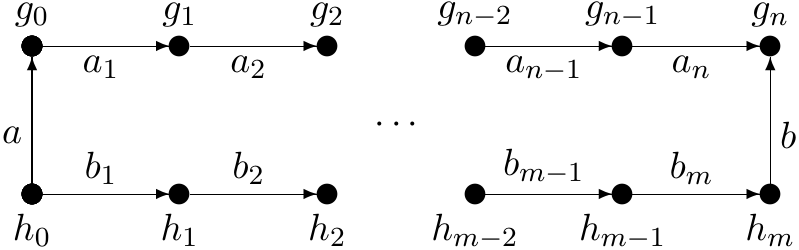}
\caption{ Two geodesics \label{rectangle}}
\end{figure}

\begin{figure}
\includegraphics[width=.5\textwidth]{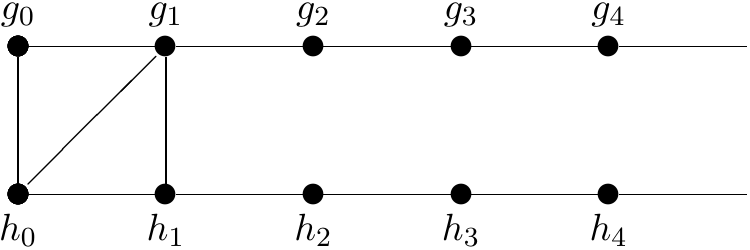}
\caption{ The case $\delta_2=1$. \label{D2333}}
\end{figure}
If $g_n$ is joined to $h_{m-1}$ or $h_m$ to $g_{n-1}$ in $D$, then we 
are done by induction on $m+n$. Thus we may assume that only $g_0$ or 
$h_0$ can have degree 2. Hence either $\delta_2=1$, and there are at 
least 5 galleries or $\delta_2=0$ and there are at least 6 galleries. As 
the vertices supporting a gallery cannot all lie on any one side, 6 
galleries is the maximum there can be. Consequently the inequality 
in Lemma~\ref{galleries} is an equality, and all internal vertices must 
have degree 6.

Suppose $\delta_2=1$. By symmetry we may assume $g_0$ has degree 2. 
Because there are at least 5 galleries, none attached to a single side, 
it must be that the other corner vertices all have degree 3. The 
situation is illustrated in Figure~\ref{D2333}.

Now consider the subdiagram $D'$ with corners $g_1$, $g_n$, $h_1$, 
$h_m$. Of course $D'$ is a diagram of minimum area for its boundary.
If either $g_1$ or $h_1$ has degree 2 with respect to this subdiagram, 
then there is an edge from $g_2$ to $h_2$, and we are done. The 
alternative is that the corners of $D'$ have degree 3, there are 6 
galleries, and all internal vertices of $D'$ have degree 6. See 
Figure~\ref{D2333P} where the degree of $g_2$ is 3 or 4 depending on 
whether or not $p=g_3$ and likewise for $h_2$ and $q$.
\begin{figure}
\includegraphics[width=.5\textwidth]{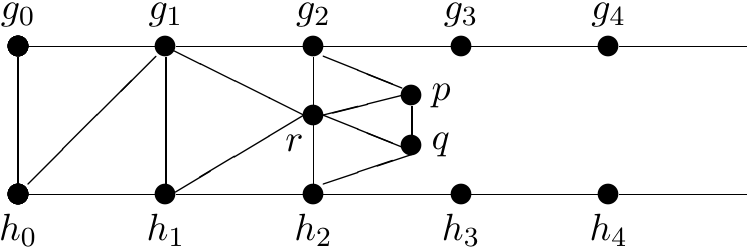}
\caption{ The case $\delta_2=1$ continued. \label{D2333P}}
\end{figure}
We know that the vertices $h_1,h_2, \ldots$ are part of a gallery in 
$D'$. That is, in $D'$,  $h_1$ and $h_k$ have degree 3 for some $k \ge 
2$ and all intervening boundary vertices have degree 4. But now 
Definition~\ref{combing} implies that we can remove the edge from $h_1$ 
to $r$ and replace it with an edge from $g_1$ to $h_2$. With this change 
we obtain a diagram of minimum area with an internal vertex of degree 5 
in contradiction to Lemma~\ref{minimal}.

It remains to consider the case in which $D$ has $\delta_2=0$. Here we 
begin with the situation of Figure~\ref{D3333} and use the preceding 
argument. 
\begin{figure}
\includegraphics[width=.5\textwidth]{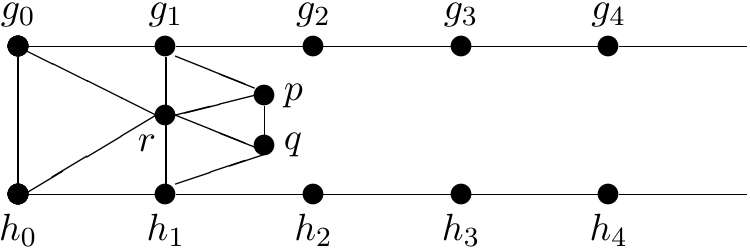}
\caption{ The case $\delta_2=0$. \label{D3333}}
\end{figure}
If $q \ne h_2$, then we replace the edge from $h_1$ to $q$ by one from 
$r$ to $h_2$ thereby decreasing the degree of $q$ to 5 and obtaining a 
contradiction. We obtain a similar contradiction if $p \ne g_2$. Finally 
if $q=h_2$ and $p=g_2$, then we are done.


\begin{thebibliography}{99}

\bibitem{Ba} R. Baer, Free sums of groups and their generalizations II, 
III, American J. Mathematics, {\bf 72} (1950), 625-670.

\bibitem{EH} M. Edjvet and J. Howie,
Star graphs, projective planes and free subgroups in small cancellation groups.
Proc.\ London Math.\ Soc.\ (3), {\bf 57} (1988), 301-328. 

\bibitem{Ev} T. Evans, Embeddability and the word problem, J. London 
Math.\ Soc., {\bf 28} (1953), 76-80.

\bibitem{GL2012} A. Gaglione, S. Lipschutz and D. Spellman,
Survey of generalized pregroups and a question of Reinhold Baer, 
Algebra Discrete Math., {\bf 13} (2012), 220-236. 

\bibitem{GS1}
S. Gersten and H. Short, Small cancellation theory and automatic groups, 
Inventiones Math., {\bf 102} (1990), 305-334.

\bibitem{J1}
A. Juhasz, Small cancellation theory with a weakened small cancellation hypothesis. I: The basic theory, Israel J. Math., {\bf 55} (1986), 65-93.

\bibitem{J2}
A. Juhasz, Small cancellation theory with a weakened small cancellation hypothesis  II: The word problem, Israel J. Math.\ {\bf 58} (1987), 19-36.

\bibitem{J3}
A. Juhasz, Small cancellation theory with a weakened small cancellation hypothesis III: The conjugacy problem, Israel J. Math. {\bf 58} (1987), 37-53. 

\bibitem{J4}
A. Juhasz, Small cancellation theory with a unified small cancellation 
condition I, J. London Math. Soc. (2), {\bf 40} (1989), 57-80.

\bibitem{MW1} J. McCammond and D. Wise,
Fans and ladders in small cancellation theory,
Proc.\ London Math.\ Soc.\ (3), {\bf 84} (2002), 599-644.

\bibitem{MW2} J. McCammond and D. Wise,
Coherence, local quasiconvexity, and the perimeter of 2-complexes.
Geom.\ Funct.\ Anal., {\bf 15} (2005), 859-927. 

\bibitem{O} A. Yu.\ Ol'shansk\u{i}, A Geometry of defining relations in groups, Translated from the 1989 Russian original by Yu.\ A. Bakhturin. Mathematics and its Applications (Soviet Series), 70. Kluwer Academic Publishers Group, 1991.

\bibitem{Ri} F. S. Rimlinger, Pregroups and Bass-Serre theory, Memoirs 
of the A. M. S., {\bf 361} 1987.

\bibitem{St1}
J. Stallings, {\em Group Theory and Three Dimensional Manifolds}, Yale 
University Press, 1971.

\bibitem{W} U. Weiss, On biautomaticity of non-homogenous small-cancellation groups, Internat.\ J. Algebra Comput.\ {\bf 17} (2007), 797-820.
\end{thebibliography}
\end{document}